\documentclass[a4paper,10pt]{article}
\AtBeginDvi{}

\title{Almost Hermitian structures on tangent bundles\thanks{This is a slightly revised version of the paper published in the Proceedings of the 11th International Workshop on Differential Geometry, Vol. 11 (2007), pp.105--118.}}
\author{Hiroyasu Satoh}
\date{}
\usepackage{amsmath,amssymb,amsthm,mathrsfs,bm,cite}

\theoremstyle{plain}
\newtheorem{theorem}{Theorem}[section]
\newtheorem{lemma}[theorem]{Lemma}

\newtheorem{proposition}[theorem]{Proposition}
\newtheorem{problem}[theorem]{Problem}
\theoremstyle{definition}
\newtheorem{definition}[theorem]{Definition}
\newtheorem{remark}[theorem]{Remark}
\newtheorem{example}[theorem]{Example}

\numberwithin{equation}{section}

\newcommand{\RN}{\mathbb{R}}
\def\tr{\mathop{\mathrm{Tr}}\nolimits}
\def\Ric{\mathop{\mathrm{Ric}}\nolimits}
\def\trp#1{\mathord{\mathopen{{\vphantom{#1}}^t}#1}}
\DeclareMathOperator*{\cycsum}{\mathfrak{S}}
\DeclareMathOperator*{\minisum}{\mbox{$\sum$}}

\begin{document}
\maketitle

\begin{abstract}
In this article, we consider the almost Hermitian structure on 
$TM$ induced by a pair of a metric and an affine connection on $M$. 
We find the conditions under which $TM$ admits almost K\"ahler 
structures, K\"ahler structures and Einstein metrics. 
Moreover, we provide two examples of K\"ahler-Einstein structures 
on $TM$. 
\vspace{8pt}

\noindent
{\it 2000 Mathematics Subject Classification}:\ 53C55, 53C15, 53C25.
\vspace{8pt}

\noindent
{\it Keywords}:
almost K\"ahler manifold, K\"ahler manifold, Einstein metric,
tangent bundle, Hessian structure, statistical model
\end{abstract}

\section{Introduction}

Let $D$ be an affine connection on a Riemannian manifold $(M^n,g)$
and $\pi : TM \rightarrow M$ the tangent bundle over $M$.
Then the connection $D$ induces the direct decomposition
\begin{equation}\label{decomp_TM}
T_\xi(TM)=H_\xi(TM)\oplus V_\xi(TM)
\end{equation}
of the tangent space at $\xi\in TM$ where $H_\xi(TM)$ is called
the horizontal subspace and $V_\xi(TM)$
the vertical subspace of $T_\xi(TM)$.
These subspaces are isomorphic to the tangent space $T_{\pi(\xi)}M$.
Under the decomposition \eqref{decomp_TM} and identifications
$H_\xi(TM), V_\xi(TM) \cong T_{\pi(\xi)}M$,
we can define an almost complex structure $J^D$ and a $J^D$-invariant
metric $\widetilde{g}^D$ which is known as the Sasaki metric,
roughly as follows;
\begin{equation*}
J^D=\left(\begin{array}{cc}O&I_n\\-I_n&O\end{array}\right),\quad
\widetilde{g}^D=g\oplus g.
\end{equation*}
We give the detailed definition of $J^D$
and $\widetilde{g}^D$ later.
We call the almost Hermitian structure
$(J^D, \widetilde{g}^D)$ the natural almost Hermitian structure on $TM$
induced by $(g, D)$.

In this article, we find the conditions under which $TM$
admits almost K\"ahler structures, K\"ahler structures and Einstein
metrics, respectively.
Moreover, we provide some examples of K\"ahler-Einstein structures on $TM$.
A reason why we investigate the Einstein condition is principally
that the Goldberg conjecture\cite{Go}, which states that
{\it a compact almost K\"ahler Einstein manifold is K\"ahler},
is still unsolved completely.
We now recall the definitions of almost K\"ahler and K\"ahler structures.
Let $(J, g)$ be an almost Hermitian structure.
If the K\"ahler form $\Omega=g(J\cdot , \cdot )$ is closed,
we say that $(J, g)$ is almost K\"ahler.
If $d\Omega=0$ and $J$ is integrable, we say that $(J, g)$ is K\"ahler.
Sekigawa\cite{SeK} proved that the Goldberg conjecture is true
when the scalar curvature is non-negative.
In the case that the scalar curvature is negative, some partial
solutions are obtained
(for example, see \cite{OS, It2}, also refer to \cite{SaH}).
It is known that the assumption about the compactness is essential for
the conjecture.
Nurowski and Przanowski\cite{NP} gave a counter-example to
non-compact version of the Goldberg conjecture by showing that $\RN^4$
admits an almost K\"ahler structure which is non-K\"ahler and Ricci flat.
The motivation of our research is to construct examples of (non-compact)
non-K\"ahler, almost K\"ahler Einstein manifolds with nonzero scalar
curvature, which remains still under investigation.

Our main theorem is the following;
\begin{theorem}\label{main}
Let $(g, D)$ be a metric and an affine connection on $M$
and $(J^D, \widetilde{g}^D)$ be the natural almost Hermitian structure
induced by $(g, D)$.
Then
\begin{enumerate}
 \item $(J^D, \widetilde{g}^D)$ is almost K\"ahler if and only if
the dual connection $D^*$ of $D$ with respect to $g$ is torsion-free.
Here the dual connection $D^*$ is the affine connection defined
by the condition
\begin{equation*}
Z\left(g(X, Y)\right) = g(D^*_Z X,Y) + g(X, D_Z Y)
\end{equation*}
for $X, Y, Z \in \mathfrak{X}(M)$
where $\mathfrak{X}(M)$ is the set of all smooth vector fields on $M$.
 \item $(J^D, \widetilde{g}^D)$ is K\"ahler if and only if 
$(M, g, D)$ is a Hessian manifold, i.e. the connection $D$ and 
its dual $D^*$ with respect to $g$ are both flat.
Here ``flat'' means that its torsion and curvature both vanish.
 \item If the Sasaki metric $\widetilde{g}^D$ on $TM$ is Einstein,
then the curvature tensor of $D$ vanishes.
\end{enumerate}
\end{theorem}

\begin{remark}
(i)\ The cotangent bundle $T^*M$ carries a canonical symplectic form
$\Omega^*$.
The condition that $D^*$ is torsion-free is equivalent to the condition
that the K\"ahler form of $(J^D, \widetilde{g}^D)$ coincides with the
pull-back of $\Omega^*$ via the natural isomorphism induced by $g$
(Theorem \ref{cTB_symp}).

(ii)\ Dombrowski \cite{Dom} shows that the almost complex structure $J^D$
on $TM$ is integrable if and only if the connection $D$ is flat.
From Dombrowski's theorem and Theorem \ref{main} (i),
we get immediately Theorem \ref{main} (ii).

Under the assumption that $D$ is flat, it is known that $(M, g, D)$
is a  Hessian manifold if and only if $(TM, J^D, \widetilde{g}^D)$ is
K\"ahler.
Theorem \ref{main} (ii) asserts that $(J^D, \widetilde{g}^D)$ is K\"ahler
under the assumption weaker than in \cite[Proposition 2.2.4]{Sh}.

(iii)\ Theorem \ref{main} (iii) is a direct consequence of
Theorem \ref{thm_Jinv}.
In case that $\nabla$ is the Levi-Civita connection of $g$, it is known
that if $(TM, \widetilde{g}^{\nabla})$ is never locally symmetric unless
$(M, g)$ is locally Euclidean (\cite[Theorem 2]{Kow}).
\end{remark}

From Theorem \ref{main}, the natural almost
Hermitian structure $(J^D, \widetilde{g}^D)$ on $TM$ induced by $(g, D)$
which is non-K\"ahler, but almost K\"ahler and Einstein exists only
on a manifold $(M, g, D)$ which satisfies that
\begin{equation}\label{AKnonK}
\left\{
\begin{array}{l}
\bullet\ \ \mbox{the torsion of $D$ never vanishes, and}\vspace{3pt}\\
\bullet\ \ 
\mbox{the dual connection $D^*$ of $D$ with respect to $g$ is flat.}
\end{array}
\right.
\end{equation}
In section 4.1, we construct a family of metrics and connections which
satisfy above two conditions.
Moreover we show that this family includes an almost K\"ahler Einstein
structure.
However this structure is not what we are looking for because the metric
is pseudo-Riemannian and the almost complex structure is integrable.

In section 4.2, we investigate under which conditions the tangent bundle $TM$ 
admits a K\"ahler-Einstein structure. 
As Theorem \ref{main} (ii) asserts, the natural almost Hermitian structure
which is K\"ahler can be constructed only on Hessian manifolds. 
As an example of Hessian manifold, we consider the statistical model
$\mathscr{P}_n^\rho$ which is a set of probability distributions on
$\RN^n$ induced by a matrix-valued linear map $\rho$ on a open subset
$U\subset\RN^m$.
For example, the tangent bundle over the manifold $\mathscr{P}_n^\rho$
which is induced by $\rho(t)=tI_n$ for $t\in \RN_{+}$ is shown
to have constant holomorphic sectional curvature and consequently
this is K\"ahler-Einstein. 
Here $I_n$ is the unit $n\times n$-matrix.
Thus there exists many Hessian manifolds induced by matrix-valued linear
maps.
Do there exist other matrix-valued linear maps $\rho$ such that
the tangent bundle of $\mathscr{P}_n^\rho$ is K\"ahler-Einstein?
We give new examples of Hessian manifolds whose tangent bundle
admits a K\"ahler-Einstein structure (Theorem \ref{Hmain}).

\section{Preliminaries}

\subsection{Definitions}

Let $(M^n, g)$ be an $n$-dimensional Riemannian manifold 
with an affine connection $D$. 
We denote the coefficients of the connection $D$ with respect to a local 
coordinate system $(U; x^1,\ldots, x^n)$ by $\{\Gamma^k_{ij}\}$;
\begin{equation*}
D_{\frac{\partial}{\partial x^i}}\frac{\partial}{\partial x^j}
=\sum_{k=1}^n\Gamma^k_{ij}\frac{\partial}{\partial x^k}.
\end{equation*}
Let $\pi : TM \rightarrow M$ be the tangent bundle over a manifold $M$. 
We define smooth functions $y^1,\ldots,y^n$ on $TM$ 
by $y^j(\xi)=\xi^j$ for
$\xi=\minisum_i \xi^i\frac{\partial}{\partial x^i}$.
Then, $(\pi^{-1}(U); x^1,\ldots, x^n,y^1,\ldots, y^n)$ is a local coordinate
system of $TM$.

For $X=\minisum X^i \frac{\partial}{\partial x^i},
\ \xi\in T_xM$, we define the
{\it horizontal lift} $X^H_\xi$ and the {\it vertical lift} $X^V_\xi$ of
$X$ at $\xi$ by
\begin{equation*}
\begin{split}
X^H_\xi&=\sum_i X^i \frac{\partial}{\partial x^i}
-\sum_{i,j,k}\Gamma^k_{ij} X^i y^j(\xi)\frac{\partial}{\partial y^k},\\
X^V_\xi&=\sum_i X^i\frac{\partial}{\partial y^i},
\end{split}
\end{equation*}
respectively.
$X^H_\xi, X^V_\xi$ are tangent vectors at $\xi\in TM$.
We set
\begin{equation*}
\begin{split}
H_\xi(TM):=& \{ X^H_\xi{\ };{\ }X\in T_xM\},\\
V_\xi(TM):=& \{ X^V_\xi{\ };{\ }X\in T_xM\},
\end{split}
\end{equation*}
and
\begin{equation*}
H(TM):=\bigcup_{\xi\in TM} H_\xi(TM),{\ }{\ }{\ }
V(TM):=\bigcup_{\xi\in TM} V_\xi(TM).
\end{equation*}
We call $H(TM)$ and $V(TM)$ the {\it horizontal} and the 
{\it vertical subbundles}, respectively. 
Then we obtain the direct decomposition of the tangent bundle over $TM$;
\begin{equation*}
T(TM)=H(TM)\oplus V(TM).
\end{equation*}

\begin{definition}
Let $(M,g)$ be a Riemannian manifold with an affine connection $D$. 
Then we define an almost complex structure $J^D$ by
\begin{equation*}
J^D X^H_\xi = X^V_\xi,{\ }{\ }{\ }J^D X^V_\xi = -X^H_\xi,
\end{equation*}
and a Riemannian metric $\widetilde{g}^D$ on $TM$, 
which is called the {\it Sasaki metric}, by
\begin{equation*}
\widetilde{g}^D(X^H_\xi, Y^H_\xi)=\widetilde{g}^D(X^V_\xi, Y^V_\xi)
=g(X,Y),{\ }{\ }{\ }\widetilde{g}^D(X^H_\xi, Y^V_\xi)=0
\end{equation*}
for $X, Y, \xi\in T_xM$. 
We call $(J^D, \widetilde{g}^D)$ the {\it natural almost  Hermitian
structure on $TM$ induced by $(g, D)$}.
\end{definition}

Now we mention the dual connection.
We recall that the {\it dual connection} $D^*$ of a connection $D$
with respect to a metric $g$ is defined by
\begin{equation}\label{def_dualconn}
Z\left(g(X, Y)\right) = g(D^*_Z X,Y) + g(X, D_Z Y).
\end{equation}
The torsion tensor $T^D$ and the curvature tensor $R^D$ of a connection
$D$ are $(1,2)$- and $(1,3)$-tensor fields respectively, defined by
\begin{gather}
T^D(X,Y)=D_X Y-D_Y X-[X,Y],\notag\\
R^D(X,Y)Z=D_X(D_Y Z)-D_Y(D_X Z)-D_{[X,Y]}Z,
\end{gather}
where $X, Y, Z\in \mathfrak{X}(M)$.
The curvature and torsion tensors of a connection $D$
and its dual $D^*$ satisfy the following relation;
\begin{gather}
g(T^{D^*}(X, Y), Z)=(D_X g)(Y, Z)-(D_Y g)(X, Z)+g(T^D(X, Y), Z),\label{dualT}\\
R^{D^*}_g(Z, W, X, Y)=-R^D_g(W, Z, X, Y)\label{dualR}
\end{gather}
for $X, Y, Z, W\in T_xM$.
Here $R^D_g$ is the $(0,4)$-tensor field defined by
\begin{equation*}
R^D_g(W, Z, X, Y)=g(R^D(X, Y)Z, W).
\end{equation*}

\subsection{The Levi-Civita connection and the curvature of the Sasaki metric}

The bracket product of horizontal and vertical vectors are determined by 
the following formulae (\cite[Lemma 2]{Dom});
\begin{equation}
\begin{split}
\left[X^H, Y^H\right]_{(\xi)}
&=([X, Y]_{(x)})^H_\xi-\left(R^D(X_{(x)}, Y_{(x)})\xi\right)^V_\xi,\\
\left[X^H, Y^V\right]_{(\xi)}
&={\left(D_X Y\right)^V}_{(\xi)},\\
\left[X^V, Y^V\right]_{(\xi)}&=0
\end{split}\label{bracket}
\end{equation}
for $X, Y\in \mathfrak{X}(M)$ and $\xi\in T_xM$.
Here $X^H$ and $X^V$ are vector fields on $TM$ defined by
\begin{equation*}
{X^H}_{(\xi)}=(X_{(x)})^H_{\xi}, ~~~~~{X^V}_{(\xi)}=(X_{(x)})^H_{\xi}.
\end{equation*}

Let $\widetilde{\nabla}$ be the Levi-Civita connection of the Sasaki 
metric $\widetilde{g}^D$ induced by $(g, D)$.
Using \eqref{bracket} and the explicit formula
\begin{multline}
2 \widetilde{g}^D(\widetilde{\nabla}_\mathcal{X} \mathcal{Y}, \mathcal{Z})
=\mathcal{X} \left(\widetilde{g}^D(\mathcal{Y}, \mathcal{Z})\right)
+\mathcal{Y} \left(\widetilde{g}^D(\mathcal{X}, \mathcal{Z})\right)
-\mathcal{Z} \left(\widetilde{g}^D(\mathcal{X}, \mathcal{Y})\right)\\
+\widetilde{g}^D([\mathcal{X}, \mathcal{Y}],\mathcal{Z})
+\widetilde{g}^D([\mathcal{Z}, \mathcal{X}],\mathcal{Y})
+\widetilde{g}^D([\mathcal{Z}, \mathcal{Y}],\mathcal{X})\notag
\end{multline}
for $\mathcal{X}, \mathcal{Y}, \mathcal{Z}\in \mathfrak{X}(TM)$,
we obtain the following.

\begin{lemma}\label{nabla}
Let $\widetilde{\nabla}$ be the Levi-Civita connection of
the Sasaki  metric $\widetilde{g}^D$ induced by $(g, D)$.
If $X, Y, Z\in \mathfrak{X}(M)$ and $\xi\in T_x M$, then
\begin{gather}
(\widetilde{\nabla}_{X^H}Y^H)_{(\xi)}
={\left(\nabla_X Y\right)^H}_{(\xi)}
-\frac{1}{2}\left(R^D(X_{(x)},Y_{(x)})\xi\right)^V_\xi,
\notag\\
\widetilde{g}^D\left(\widetilde{\nabla}_{X^V}Y^H, Z^H\right)_{(\xi)}
=\widetilde{g}^D\left(\widetilde{\nabla}_{Y^H}X^V, Z^H\right)_{(\xi)}
=\frac{1}{2}R^D_g(X_{(x)},\xi,Y_{(x)},Z_{(x)}),\notag\\
\widetilde{g}^D\left(\widetilde{\nabla}_{X^V}Y^H, Z^V\right)_{(\xi)}
=\frac{1}{2}(D_Y g)(Z,X)_{(x)},\notag\\
\widetilde{g}^D\left(\widetilde{\nabla}_{X^H}Y^V, Z^V\right)_{(\xi)}
=g(D_X Y, Z)_{(x)}+\frac{1}{2}(D_X g)(Y,Z)_{(x)},\notag\\
\widetilde{g}^D\left(\widetilde{\nabla}_{X^V}Y^V, Z^H\right)_{(\xi)}
=-\frac{1}{2}(D_Z g)(X, Y)_{(x)},\notag\\
\widetilde{g}^D\left(\widetilde{\nabla}_{X^V}Y^V, Z^V\right)_{(\xi)}=0.
\notag
\end{gather}
\end{lemma}

We shall give the formulae of the curvature tensor of the Sasaki metric
$\widetilde{g}^D$.

\begin{proposition}
Let $\widetilde{R}$ be the Riemannian curvature tensor of
$\widetilde{g}^D$.
If $X, Y, Z, W, \xi \in T_xM$, then
\begin{equation}\label{RHHHH}
\begin{split}
\widetilde{R}_{\widetilde{g}^D}(Z^H_\xi, &W^H_\xi,X^H_\xi, Y^H_\xi)\\
=&R^{\nabla}_g(Z, W, X, Y)
-\frac{1}{2}R^D_g(R^D(Z,W)\xi, \xi, X, Y)\\
&-\frac{1}{4}\left\{R^D_g(R^D(X, Z)\xi, \xi, Y, W)
-R^D_g(R^D(Y, Z)\xi, \xi, X, W)\right\},
\end{split}
\end{equation}
\begin{equation}
\begin{split}
\widetilde{R}_{\widetilde{g}^D}(Z^H_\xi, W^V_\xi,X^V_\xi, Y^V_\xi)
=\frac{1}{4}\sum_i &\left\{  (D_{e_i} g)(Y,W){\ }R^D(X, \xi, Z, e_i)\right.\\
&{\ }{\ }{\ }-\left.(D_{e_i} g)(X,W){\ }R^D(Y, \xi, Z, e_i)\right\},
\end{split}
\end{equation}
\begin{equation}
\begin{split}
\widetilde{R}_{\widetilde{g}^D}(Z^V_\xi, W^V_\xi,X^V_\xi, Y^V_\xi)
=-\frac{1}{4}\sum_i & \left\{(D_{e_i} g)(X,Z){\ }(D_{e_i} g)(Y,W)\right.\\
&{\ }{\ }{\ }-\left.(D_{e_i} g)(Y,Z){\ }(D_{e_i} g)(X,W)\right\},
\end{split}
\end{equation}
\begin{equation}
\begin{split}
\widetilde{R}_{\widetilde{g}^D}(Z^H_\xi, &W^V_\xi, X^H_\xi, Y^H_\xi)\\
=&\frac{1}{2}R^D_g(Z, \xi, W, T^D(X, Y))
+\frac{1}{2}(D_W g)(Z, R^D(X,Y)\xi)\\
&+\frac{1}{2}\left\{
R^D_g(Z, \xi, \gamma(X, W), Y)-R^D_g(Z, \xi, \gamma(Y, W), X)
\right\}\\
&-\frac{1}{4}\left\{
(D_Y g)(Z, R^D(W, X)\xi)-(D_X g)(Z, R^D(W, Y)\xi)
\right\}\\
&+\frac{1}{2}g\left(
(D_X R^D)(Y, W)\xi-(D_Y R^D)(X, W)\xi, Z
\right),
\end{split}
\end{equation}
\begin{equation}
\begin{split}
\widetilde{R}_{\widetilde{g}^D}(Z^V_\xi, W^V_\xi, X^H_\xi, Y^H_\xi)
=&\frac{1}{2}\left\{
R^D_g(Z, W, X, Y)-R^D_g(W, Z, X, Y)
\right\}\\
&-\frac{1}{4}\sum_i\left\{
R^D_g(Z, \xi, X, e_i){\ }R^D_g(W, \xi, Y, e_i)\right.\\
&\hspace{43pt}-\left.R^D_g(W, \xi, X, e_i){\ }R^D_g(Z, \xi, Y, e_i)\right\}\\
&-\frac{1}{4}\sum_i \left\{(D_X g)(Z, e_i){\ }(D_Y g)(W, e_i)\right.\\
&\hspace{43pt}-\left.(D_X g)(W, e_i){\ }(D_Y g)(Z, e_i)\right\},
\end{split}
\end{equation}
\begin{equation}\label{RHVHV}
\begin{split}
\widetilde{R}_{\widetilde{g}^D}(Z^H_\xi, W^V_\xi, X^H_\xi, Y^V_\xi)
=&\frac{1}{2}R^D_g(W, Y, Z, X)\\
&-\frac{1}{2}(D^2_{XZ} g)(Y, W)
-\frac{1}{2}(D_{\gamma(X, Z)} g)(Y, W)\\
&+\frac{1}{4}
\sum_i R^D_g(W, \xi, X, e_i){\ }R^D_g(Y, \xi, Z, e_i)\\
&+\frac{1}{4}\sum_i (D_X g)(W, e_i){\ }(D_Z g)(Y, e_i).
\end{split}
\end{equation}
Here $\{e_i\}$ is an orthonormal basis of $T_x M$ and
$\gamma^D$ is the difference between the connection $D$ and $\nabla$;
\begin{equation*}
\gamma^D(X, Y)=D_X Y-\nabla_X Y.
\end{equation*}
\end{proposition}

A proof of the above theorem is given by direct calculations using 
Lemma \ref{nabla}. 

\begin{proposition}\label{RicSas}
The Ricci tensor $\widetilde{\Ric}$ of $\widetilde{g}^D$, defined by
\begin{equation*}
\widetilde{\Ric}(\mathcal{X}, \mathcal{Y})
=\tr \left\{\mathcal{Z}\mapsto 
\widetilde{R}(\mathcal{Z}, \mathcal{Y})\mathcal{X}\right\},
\end{equation*}
is given by the following formulae:
\begin{equation}\label{RicHH}
\begin{split}
\widetilde{\Ric}(X^H_\xi, Y^H_\xi)
=&\Ric^{\nabla}(X, Y)
-\frac{1}{2}\sum_i R^D_g(R^D(X,e_i)\xi,\xi,Y,e_i)\\
&-\frac{1}{4}\sum_i 
\left( (D^2_{XY}+D^2_{YX}+D_{\gamma(X,Y)+\gamma(Y,X)} )g\right) (e_i, e_i)\\
&+\frac{1}{4}\sum_i (D_X g)(e_i, e_j){\ }(D_Y g)(e_i, e_j),
\end{split}
\end{equation}
\begin{equation}\label{RicVV}
\begin{split}
\widetilde{\Ric}(X^V_\xi, Y^V_\xi)
=&\frac{1}{4}\sum_{i,j} R^D_g(X, \xi, e_i,e_j){\ }R^D_g(Y,\xi,e_i,e_j)\\
&-\frac{1}{2}\sum_i \left\{(D^2_{e_i e_i} g)(X, Y)
-(D_{e_i} g)(X, Y)\tr(T^{D^*})(e_i)\right\}\\
&+\frac{1}{2}\sum_{i,j}(D_{e_i} g)(X, e_j){\ }(D_{e_i} g)(Y,e_j),\\
\end{split}
\end{equation}
\begin{equation}\label{RicHV}
\begin{split}
\widetilde{\Ric}(X^H_\xi, &Y^V_\xi)\\
=&\frac{1}{2}\sum_i\left\{R^D_g(Y,\xi,X,e_i){\ }\tr (\gamma^D)(e_i)
-R^D_g(Y, \xi,  e_i, \gamma^D(e_i, X))\right\}\\
&-\frac{1}{2}\sum_i \left\{g\left( (D_{e_i} R^D)(e_i, X)\xi, Y\right)
+(D_{e_i} g)(R^D(e_i, X)\xi, Y)\right\}\\
&+\frac{1}{4}\sum_{i,j} R^D_g(Y,\xi,X,e_i){\ }(D_{e_i} g)(e_j, e_j).
\end{split}
\end{equation}
Here $\Ric^{\nabla}$ is the Ricci tensor of the Levi-Civita connection
$\nabla$ of $g$ and $\tr(\gamma^D)$ is a 1-form defined by
\begin{equation*}
\tr(\gamma^D)(X)=\tr\{Z\mapsto \gamma^D(Z,X)\}.
\end{equation*}
\end{proposition}

\begin{proof}
Using an orthonormal frame $\{E_i\}$ on $TM$, 
we can express $\widetilde{\Ric}$ by
\begin{equation*}
\widetilde{\Ric}(\mathcal{X},\mathcal{Y})=\sum_{i=1}^{2n}
\widetilde{R}_{\widetilde{g}^D}(E_i, \mathcal{X}, E_i, \mathcal{Y}).
\end{equation*}
If $\{e_i\}$ is an orthonormal basis of $T_x M$ with respect to $g$, 
then for $\xi\in T_xM$ $\{{e_i}^H_\xi, {e_i}^V_\xi\}$ is an orthonormal basis 
of $T_\xi(TM)$ with respect to $\widetilde{g}^D$. 
Hence we have
\begin{equation}\label{HHex}
\begin{split}
\widetilde{\Ric}(X^H_\xi, Y^H_\xi)
=\sum_{i=1}^n&\left\{
\widetilde{R}_{\widetilde{g}^D}({e_i}^H_\xi,X^H_\xi,{e_i}^H_\xi,Y^H_\xi)
\right.\\
&{\ }{\ }{\ }+\left.
\widetilde{R}_{\widetilde{g}^D}({e_i}^V_\xi,X^H_\xi,{e_i}^V_\xi,Y^H_\xi)
\right\}.
\end{split}
\end{equation}
Substituting \eqref{RHHHH} and \eqref{RHVHV} into \eqref{HHex}, 
we obtain \eqref{RicHH}. 
Similarly, by simple calculations we obtain \eqref{RicVV} and \eqref{RicHV}.
\end{proof}

\section{Proof of Theorem \ref{main}}

(i)\ Using \eqref{dualT}, \eqref{dualR}, \eqref{bracket} and the formula
\begin{equation*}
d\Omega(\mathcal{X}, \mathcal{Y}, \mathcal{Z})
=\cycsum_{\mathcal{X}, \mathcal{Y}, \mathcal{Z}}
\left\{
\mathcal{X}\left(\Omega(\mathcal{Y}, \mathcal{Z})\right)
-\Omega([\mathcal{X}, \mathcal{Y}], \mathcal{Z})\right\},
\end{equation*}
we have
\begin{align}
d\Omega(X^H_\xi, Y^H_\xi, Z^H_\xi)
&=\cycsum_{X,Y,Z} R^{D^*}_g(\xi, X, Y, Z),\label{AK1}\\
d\Omega(X^H_\xi, Y^H_\xi, Z^V_\xi)
&=g(T^{D^*}(X, Y), Z),\label{AK2}\\
d\Omega(X^H_\xi, Y^V_\xi, Z^V_\xi)
&=d\Omega(X^V_\xi, Y^V_\xi, Z^V_\xi)=0.\notag
\end{align}
Here $\displaystyle \cycsum_{X, Y, Z}$ denotes the cyclic sum
with respect to $X, Y, Z$.

If we assume $d\Omega=0$, from \eqref{AK2} we obtain that $T^{D^*}=0$. 
Conversely, if $D^*$ is torsion-free, then from the first Bianchi identity,
we find that the right-hand side of \eqref{AK1} vanishes. 
This completes the proof of Theorem \ref{main} (i).

Now we remark about the symplectic structure on $T^*M$.
Let $\pi^* : T^*M\to M$ be the cotangent bundle on $M$.
We define smooth functions $z_1,\ldots,z_n$ by $z_i(\psi)=\psi_i$
on $T^*M$ for $\psi=\sum \psi_i dx^i \in T_x^*M$.
Then, $(x^1,\ldots,x^n,z_1,\ldots,z_n)$ is a local coordinate system
of $T^*M$.
$T^*M$ carries a canonical symplectic structure $\Omega^*$ locally
expressed by $\Omega^* =\sum dx^i\wedge dz_i$
(See \cite{ACdS}).

Then, we obtain the following result.

\begin{theorem}\label{cTB_symp}
The natural almost Hermitian structure $(J^D, \widetilde{g}^D)$ on $TM$
induced by $(g, D)$ is almost K\"ahler if and only if the K\"ahler form
of $(J^D, \widetilde{g}^D)$ coincides with the pull-back of
the symplectic form $\Omega^*$ on $T^*M$ by $\varphi_g$.
Here $\varphi_g : TM \rightarrow T^*M$ is the natural isomorphism
defined by $\varphi_g(X)=g(X, \cdot )$ for $X\in TM$.
\end{theorem}

\begin{proof}
Let $\varphi^*_g(\Omega^*)$ be the pull-back of the symplectic form 
$\Omega^*$ by $\varphi_g$.
Straightforward computations show that
\begin{equation}\label{pb_cT}
\begin{split}
\varphi_g^*(\Omega^*)
\left(\frac{\partial}{\partial x^i},\frac{\partial}{\partial x^j}\right)
&=\sum_k\left(
\frac{\partial g_{ik}}{\partial x^j}-\frac{\partial g_{jk}}{\partial x^i}
\right)y^k,\\
\varphi_g^*(\Omega^*)
\left(\frac{\partial}{\partial x^i},\frac{\partial}{\partial y^j}\right)
&=g_{ij},\\
\varphi_g^*(\Omega^*)
\left(\frac{\partial}{\partial y^i},\frac{\partial}{\partial y^j}\right)
&=0,
\end{split}
\end{equation}
and
\begin{equation}\label{K_T}
\begin{split}
\Omega\left(\frac{\partial}{\partial x^i},\frac{\partial}{\partial x^j}\right)
&=\sum_{k,l}(\Gamma^l_{jk}g_{li}-\Gamma^l_{ik}g_{lj})y^k,\\
\Omega\left(\frac{\partial}{\partial x^i},\frac{\partial}{\partial y^j}\right)
&=g_{ij,}\\
\Omega\left(\frac{\partial}{\partial y^i},\frac{\partial}{\partial y^j}\right)
&=0.
\end{split}
\end{equation}
Hence, from \eqref{pb_cT} and \eqref{K_T},
$\varphi_g^*(\Omega^*)=\Omega$ if and only if
\begin{equation}\label{pb_cond}
\frac{\partial g_{ik}}{\partial x^j}-\frac{\partial g_{jk}}{\partial x^i}
=\sum_{l}\left(\Gamma^l_{jk}g_{li}-\Gamma^l_{ik}g_{lj}\right).
\end{equation}
From \eqref{dualT}, we find that \eqref{pb_cond} implies $T^*=0$.
\end{proof}
\vspace{5pt}

(ii)\ 
The integrability condition of the almost complex structure is equivalent
to the condition that the Nijenhuis tensor $N$ vanishes
(See \cite[Chapter IX]{KN2}).
Here $N$ is a $(1, 2)$-tensor field defined by
\begin{equation*}
N(\mathcal{X},\mathcal{Y})=[J^D\mathcal{X}, J^D\mathcal{Y}]
-J^D[J^D\mathcal{X}, \mathcal{Y}]-J^D[\mathcal{X}, J^D\mathcal{Y}] 
-[\mathcal{X},\mathcal{Y}]
\end{equation*}
for $\mathcal{X}, \mathcal{Y}\in \mathfrak{X}(TM)$. 
From the definition, $N$ satisfies
\begin{equation*}
N(\mathcal{Y}, \mathcal{X})=-N(\mathcal{X}, \mathcal{Y}),\hspace{20pt}
N(J^D\mathcal{X}, \mathcal{Y})=-J^DN(\mathcal{X}, \mathcal{Y}).
\end{equation*}
Hence, in our situation, in order to find the condition that $N=0$
it is enough to compute $N(X^H_\xi, Y^H_\xi)$.
From straightforward computations, we have
\begin{equation*}
N(X^H_\xi, Y^H_\xi)=\left(T(X,Y)\right)^H_\xi+\left(R(X, Y)\xi\right)^V_\xi.
\end{equation*}
Hence, $N=0$ if and only if $R^D=0$ and $T^D=0$, i.e. $D$ is a 
flat connection \cite{TO}.
From the above argument and Theorem \ref{main} (i), 
we obtain Theorem \ref{main} (ii).
\vspace{5pt}

(iii)\ We show the following fact;
\begin{theorem}\label{thm_Jinv}
If the Ricci tensor $\widetilde{\Ric}$ of $\widetilde{g}^D$ is
$J^D$-invariant or $J^D$-anti-invariant, then $R^D=0$.
\end{theorem}

\begin{proof}
The assumption implies that $\widetilde{\Ric}$ satisfies
\begin{equation}\label{E-condi}
\widetilde{\Ric}(X^H_\xi, Y^H_\xi)=\pm \widetilde{\Ric}(X^V_\xi, Y^V_\xi)
\end{equation}
for any $X, Y, \xi \in TM$, and in particular
\begin{equation}\label{E-condi2}
-\frac{1}{2}\sum_i R^D_g(R^D(X,e_i)\xi,\xi,Y,e_i)
=\pm
\frac{1}{4}\sum_{i,j} R^D_g(X,\xi,e_i,e_j){\ }R^D_g(Y,\xi,e_i,e_j)
\end{equation}
which is the $\xi$-dependent part of \eqref{E-condi}. 
Substituting $X=Y=e_i$ and $\xi=e_j$ into \eqref{E-condi2}, 
and summing on $i$ and $j$ we have
\begin{equation*}
-\frac{1}{2}|R^D|^2_g=\pm \frac{1}{4}|R^D|^2_g
\end{equation*}
from which we obtain that $|R^D|^2_g=0$, i.e. $R^D=0$.
\end{proof}

If $(TM, J^D, \widetilde{g}^D)$ is Einstein, the Ricci tensor is 
$J^D$-invariant. 
Hence, Theorem \ref{main} (iii) is obtained as the corollary of 
Theorem \ref{thm_Jinv}.

\section{Examples}

\subsection{A 1-parameter family of almost K\"ahler structures
on the tangent bundle}

Let $(M, g)$ be the Riemannian product of the unit circle $(S^1, g_0)$
with the angular coordinate $\theta$
and a space of positive constant curvature $(N^{n-1}, g_N)$
and let $\omega=k d\theta$ be a 1-form on $M$ ($k\in \RN$).
Also let $D$ be a torsion-free connection $D$ on $M$ satisfying
$Dg=\omega\otimes g$.
Such a connection is uniquely determined for given $g$ and $\omega$.

When $\displaystyle k=\pm\frac{2s_N}{\sqrt{(n-1)(n-2)}}$,
the curvature of $D$ vanishes, i.e. $D$ is a flat connection.
Here $s_N$ is the scalar curvature of $g_N$.

Fix a constant $k$ such that $R^D=0$.
For $\lambda \in \RN$ we define the metric $g_\lambda$ on $M$ by 
\begin{equation*}\label{def_nm}
g_\lambda:=g+\frac{\lambda}{|\omega|^2_g}\omega\otimes \omega.
\end{equation*}
Let $D^*_{\lambda}$ be the dual connection of $D$ with respect to 
$g_\lambda$. 
Then $(g_{\lambda}, D^*_\lambda)$ induces a 1-parameter family of 
almost K\"ahler structures $(J_\lambda, \widetilde{g}_\lambda)$ on $TM$
parametrized by $\lambda$.
Moreover, $(g_{\lambda}, D^*_\lambda)$ satisfies the condition
\eqref{AKnonK}.

Now we find the condition that $\widetilde{g}^D$ is Einstein.
In our situation, from Proposition \ref{RicSas}
we can express $\widetilde{\Ric}$ by
\begin{equation}\label{WRicHH}
\begin{split}
\widetilde{\Ric}(X^H_\xi,Y^H_\xi)
=&\frac{|\omega|^2_g}{8}\left(2(n-2)-\frac{\lambda^2}{\lambda+1}\right)
g_\lambda(X,Y)\\
&+\frac{1}{8}(\lambda+2)\left\{\lambda-2(n-1)\right\}\omega(X)\omega(Y),
\end{split}
\end{equation}
\begin{equation}\label{WRicVV}
\widetilde{\Ric}(X^V_\xi,Y^V_\xi)
=\frac{|\omega|^2_g}{8}\frac{\lambda^2+2\lambda-2n}{\lambda+1}
g_\lambda(X,Y)-\frac{n\lambda(\lambda+2)}{8}\omega(X)\omega(Y),
\end{equation}
\begin{equation*}
\widetilde{\Ric}(X^H_\xi,Y^V_\xi)=0.
\end{equation*}
From \eqref{WRicHH} and \eqref{WRicVV},
$\widetilde{\Ric}=k{\ }\widetilde{g}_{\lambda}$ 
if and only if $\lambda=-2$. 
Then, $D^*_{(-2)}$ is torsion free
and $g_{(-2)}$ is a pseudo-Riemannian metric. 
Hence $(TM, J_{(-2)}, \widetilde{g}_{(-2)})$ is a pseudo-K\"ahler Einstein
manifold.

\begin{remark}
We can apply a similar argument to compact flat Weyl manifolds
(See \cite{SaH2} for details).
\end{remark}

At the end of this subsection we pose the following problem;

\begin{problem}
Do there exist pairs $(g, D)$ where $g$ is a positive definite metric
and $D$ is an affine connection such that the natural almost Hermitian
structure $(J^D, \widetilde{g}^D)$ is strictly almost K\"ahler 
Einstein?
\end{problem}

\subsection{On the manifold of 
multivariate normal distributions on $\RN^2$}

In this subsection, we consider K\"ahler structures on $TM$.

From Theorem \ref{main}, we can construct K\"ahler structures on tangent
bundles by using Hessian structures.
We recall that when an affine connection $D$
and its dual connection $D^*$ with respect to $g$ are both flat,
we call $(M, g, D)$ a {\it Hessian manifold}, and in particular $g$
a {\it Hessian metric} on $(M, D)$.
This condition is equivalent to the following condition:
$D$ is a flat connection on $M$ and there exists a function $\varphi$ on
$M$ such that $g=Dd\varphi$.
We call the function $\varphi$ the {\it potential} of the Hessian 
metric $g$ with respect to $D$.

We define a 1-form $\alpha$ and a symmetric $(0,2)$-tensor $\beta$, 
which are called the {\it first Koszul form} 
and the {\it second Koszul form} respectively, by
\begin{equation*}
\alpha(X)=g(D_Xdv_g, dv_g)\ \ (X\in TM),\quad \beta=D\alpha.
\end{equation*}
Here $dv_g$ is the volume form of $g$.
Then we can consider the notion of ``Hesse-Einstein''.
If the 2nd Koszul form is proportional to the Hessian metric $g$,
then we say that a Hessian manifold is {\it Hesse-Einstein}.

\begin{remark}
If $\dim M\ge 2$ and $\beta=c g$, 
then $c$ must be a constant.
\end{remark}

Hesse-Einstein manifolds are characterized as follows:
\begin{theorem}[{\cite[Thm. 3.1.6]{Sh}}]
Let $(M, g, D)$ be a Hessian manifold. 
Then the Ricci tensor $\widetilde{\Ric}$ of $\tilde{g}^D$ satisfies
\begin{equation*}
\widetilde{\Ric}(X^H_\xi,Y^H_\xi)=\widetilde{\Ric}(X^V_\xi,Y^V_\xi)=
-\beta(X,Y),\quad
\widetilde{\Ric}(X^H_\xi,Y^V_\xi)=0.
\end{equation*}
In particular, the K\"ahler structure on $TM$ induced by $(g, D)$ 
is Einstein if and only if $(M, g, D)$ is Hesse-Einstein. 
\end{theorem}

Using a matrix-valued linear map we can construct a Hessian structure
as follows:
Let $\mathcal{S}_n$ be the set of all symmetric $n\times n$-matrices 
and $\mathcal{S}^+_n$ the subset of all positive definite symmetric
matrices in $\mathcal{S}_n$.
Let $\rho$ be a linear injection 
from a domain $U\subset\RN^m$ to $\mathcal{S}_n$ which satisfies that 
$\rho(U)\subset \mathcal{S}^+_n$.
For $(\mu, \xi)\in \RN^m\times U$, we set a function
\begin{equation}\label{def_rhoPD}
p(x; \mu, \xi)
:=\sqrt{\frac{\det \rho(\xi)}{(2\pi)^n}}
\exp \left\{
-\frac{\trp{(x-\mu)} \rho(\xi) (x-\mu)}{2}
\right\}\quad (x\in \RN^n).
\end{equation}
Then the set $\mathscr{P}_n^\rho=\{p(x; \mu, \xi){\ }|{\ }
(\mu, \xi)\in \RN^n\times U\}$ is a family of probability distributions
on $\RN^n$ parametrized by $(\mu, \xi)$. 
We call $\mathscr{P}_n^\rho$ the {\it statistical model induced by $\rho$}.
$\mathscr{P}_n^\rho$ is a smooth manifold of dimension $(n+m)$ and
has the Riemannian metric $g_\rho$ which is called the Fisher metric
(See \cite{AN}).
Moreover $\mathscr{P}_n^\rho$ admits a flat connection $D$ such that
$(g_\rho, D)$ is a Hessian structure on $\mathscr{P}_n^\rho$.

\begin{proposition}[{\cite[Prop. 6.2.1]{Sh}}]\label{PP_pot}
Let $\mathscr{P}_n^\rho=\{p(x; \mu, \xi){\ }|{\ }
(\mu, \xi)\in \RN^n\times U\}$ be the statistical model induced 
by $\rho$.
We set $\theta=\rho(\xi)\mu$.
Let $D$ is the standard flat connection on
$\{ (\theta, \xi)\in \RN^n\times U\}$.
Then the Fisher metric $g_\rho$ on 
$\mathscr{P}^\rho_n$ is the Hessian metric on $(\mathscr{P}_n^\rho, D)$ 
whose potential is given by
\begin{equation*}
\varphi(\theta, \xi)=\frac{1}{2}
\{\trp{\theta}{\rho(\xi)}^{-1}\theta-\log \det \rho(\xi)\}.
\end{equation*}
\end{proposition}

\begin{example}
For the linear injection $\rho:\RN_+\rightarrow\mathcal{S}_n$ defined by
\begin{equation*}
\rho(t)=t I_n{\ }{\ }(\mbox{$I_n$ is the unit matrix}),
\end{equation*}
$(T\mathscr{P}^\rho_n, J^{D^*},\widetilde{g_\rho}^{D^*})$ has constant 
holomorphic sectional curvature and consequently
this is K\"ahler-Einstein (\cite[Problem 6.2.1]{Sh}).
Here $D^*$ is the dual connection of $D$ with respect to the Fisher 
metric $g_\rho$.
In the case that $n=1$, Sato \cite{SaT} explicitly computes the Ricci tensor of 
$(T\mathscr{P}^\rho_1, J^{D^*},\widetilde{g_\rho}^{D^*})$.
\end{example}

Then, the following problem naturally arises;
\begin{problem}
How many Hessian manifolds which are Hesse-Einstein do there
exist in the class of Hessian manifolds induced by matrix-valued
linear injections mentioned above?
\end{problem}
We give a partial solution for this problem as follows.

\begin{theorem}\label{Hmain}
For any linear injection $\rho$ from a domain $U\subset \RN^2$ into 
$\mathcal{S}_2$ such that $\rho(U)\subset \mathcal{S}^+_2$, 
$(\mathscr{P}_2^\rho, g_\rho,D^*)$ is always Hesse-Einstein. 
Hence, $(T\mathscr{P}_2^\rho, J^{D^*}, \widetilde{g_\rho}^{D^*})$ 
is a K\"ahler Einstein manifold. 
Here $D^*$ is the dual connection of $D$ with respect to $g_\rho$.
\end{theorem}

\begin{proof}[Outline of proof]
In our situation, applying certain coordinate (affine) transformations, 
we can reduce the linear injection $\rho$ into the form
\begin{equation}
\rho(\xi)=\rho(\xi^1,\xi^2)=\left(\begin{array}{cc}
\xi^1&a\xi^1+b\xi^2\\a\xi^1+b\xi^2&\xi^2\end{array}\right)
{\ }{\ }{\ }(a,b\in\RN)\label{rhoform}
\end{equation}
as follows:

In general, a linear map $\rho : U (\subset\RN^2)\rightarrow \mathcal{S}_2$
is written by
\begin{equation*}
\rho(\xi_1,\xi_2)=\left(\begin{array}{cc}
\rho_{11}(\xi_1, \xi_2)&\rho_{12}(\xi_1, \xi_2)\\
\rho_{21}(\xi_1, \xi_2)&\rho_{22}(\xi_1, \xi_2)\end{array}\right)
\end{equation*}
where $\rho_{ij}(\xi_1, \xi_2)$ is a polynomial of degree $1$
$(1\le i, j\le 2)$.

Case 1:\ If $\rho_{11}\ne c \rho_{22}$ for any constant $c$,
change the coordinate
$(\xi_1, \xi_2; \theta)\mapsto
(\rho_{11}(\xi_1, \xi_2), \rho_{22}(\xi_1, \xi_2); \theta)$.

Case 2:\ If $\rho_{11}= c \rho_{22}$ for a constant $c$,
change the coordinate
$(\xi; \theta)\mapsto (\xi; A^{-1}\theta)$ for $A\in GL(2,\RN)$
such that $\rho^{\prime}(\xi):= A\rho(\xi)\trp{A}$ satisfies
the condition of the case 2 and we consider the statistical model induced
by $\rho^{\prime}$.
Thus, we can reduce this case into the case 1.

Computing the 2nd Koszul form $\beta^*$ of the Hessian structure
$(D^*, g_{\rho})$ where $\rho$ is given in the form \eqref{rhoform},
we get
\begin{equation*}
\beta^*=3g_{\rho}
\end{equation*}
from which we obtain Theorem \ref{Hmain}.
\end{proof}

\end{document}